\documentclass[11pt]{article}
\usepackage{indentfirst}
\usepackage{latexsym}
\usepackage{hyperref}
\usepackage{graphicx}
\usepackage{mathrsfs}
\usepackage{bookmark}
\usepackage{enumerate}
\usepackage{amsfonts,amsthm,amssymb,mathtools}
\usepackage{amsmath}
\usepackage{latexsym, euscript, epic, eepic, color}
\usepackage{amssymb}

\usepackage{enumerate}
\usepackage[all]{xy}
\usepackage{setspace}
\allowdisplaybreaks

\usepackage{epstopdf}
\numberwithin{equation}{section}
\newtheorem{theorem}{Theorem}[section]

\newtheorem{question}[theorem]{Question}

\newtheorem{lemma}[theorem]{Lemma}

\newtheorem{remark}{Remark}

 \allowdisplaybreaks

\begin{document}
\baselineskip=16pt

\title{Resistance distance-based graph invariants and  spanning trees of  graphs derived from the strong product of $P_2$ and $C_n$\footnote{This work was supported by the National Natural Foundation of China [61773020]. Corresponding author: Yingui Pan(panygui@163.com)
 }
}

\author{Yingui Pan$^{a}$,   Jianping Li$^{a}$   \\
	\small  $^{a}$College of Liberal Arts and Sciences, National University of Defense Technology, \\
	\small  Changsha, China, 410073.\\
}

\date{\today}

\maketitle

\begin{abstract}
Let $G_n$ be a graph obtained by the strong product of $P_2$ and $C_n$, where $n\geqslant3$. In this paper, explicit expressions for the Kirchhoff index,  multiplicative degree-Kirchhoff index and number of spanning trees of $G_n$ are determined, respectively. It is surprising to find that the Kirchhoff (resp. multiplicative degree-Kirchhoff) index of $G_n$ is almost one-sixth of its Wiener (resp. Gutman) index. Moreover, let $\mathcal{G}^r_n$ be the set of subgraphs obtained from $G_n$ by deleting any $r$ vertical edges of $G_n$, where $0\leqslant r\leqslant n$. Explicit formulas for the Kirchhoff index and the number of spanning trees for any graph $G^r_n\in \mathcal{G}^r_{n}$ are completely established, respectively. Finally, it is interesting to see that the Kirchhoff index of $G^r_n$ is almost one-sixth of its Wiener index. 
\end{abstract}



\section{Introduction}
Let $G$ be a simple connected graph with vertex set $V(G)$ and edge set $E(G)$. Then $|V(G)|$ and $|E(G)|$ are called the \textit{order} and the \textit{size} of $G$, respectively. For a simple graph $G$ of order $n$, its \textit{adjacent matrix} $A(G)=(a_{ij})_{n\times n}$ is a $(0,1)$-matrix with $a_{ij}=1$ if and only if two vertices $i$ and  $j$ are adjacent in $G$, and the matrix $D(G):=\textrm{diag}(d_{1},d_{2},\ldots,d_{n})$ is called the diagonal matrix of vertex degrees, where $d_{i}$ is the degree of vertex $i$  in $G$ for $1\leqslant i\leqslant n$. Then the \textit{ Laplacian matrix} of $G$ is defined as ${L}(G)=D(G)-A(G)$, whereas the \textit{normalized Laplacian matrix} of $G$ is defined to be  $\mathcal{L}(G)=D(G)^{-\frac{1}{2}}L(G)D(G)^{-\frac{1}{2}}$. It should be stressed  that ${(d_{i})}^{-\frac{1}{2}}=0$ for $d_i=0$  \cite{006}. Hence, we have
\begin{align*}
&\hspace{1.5cm}(\mathcal{L}(G))_{ij}=\left\{ 
\begin{array}{ll}
{1,} &\textrm{if $i=j$};\\
{-\frac{1}{\sqrt{d_{i}d_{j}}}}, &\textrm{if $i\ne j$ and $v_i\sim v_j$};\\
{0,} & \textrm{otherwise}.\\
\end{array} 
\right. &
\end{align*}


Let $d_{ij}$ denote the \textit{distance} between two vertices $i$ and $j$ in $G$. The \textit{Wiener index} of $G$, introduced in \cite{020}, is defined as $W(G)=\sum_{i<j}d_{ij}$. Later, Gutman \cite{000e} gave a weighted version of the Wiener index which is defined as  $\textrm{Gut}(G)=\sum_{i<j}d_id_jd_{ij}$ and now known as \textit{Gutman index}. If $G$ is an $n$-vertex tree, Gutman \cite{000e} proved that $\textrm{Gut}(G)=4W(G)-(2n-1)(n-1)$.


In 1993, Klein and Randi\'c \cite{013} proposed the concept of \textit{resistance distance} by considering the electronic network.  The resistance distance $r_{ij}$ is the effective resistance between two vertices $i$ and $j$ after every edge of a graph $G$ was putted one unit resistor. This parameter is intrinsic to both  graph theory and mathematical chemistry. Similar to the definition of the Wiener index, Klein and Randi\'c \cite{013} defined  $K\!f(G):=\sum_{i<j}r_{ij}$ to be the Kirchhoff index of $G$. They also found that $K\!f(G)\leqslant{W}(G)$ with equality if and only if $G$ is a tree. For a  connected graph $G$ of order $n$, Klein \cite{011} and Lov$\acute{a}$sz \cite{015}, independently, obtained that 
$${K\!f}(G)=n\sum^{n}_{i=2}\frac{1}{\rho_i},$$	
where $0=\rho_1<\rho_2\leqslant\cdots\leqslant\rho_n$ $(n\geqslant2)$ are the eigenvalues of $L(G)$.


In 2007, Chen and Zhang \cite{005}  proposed a weighted version of the Kirchhoff index, which is defined as $K\!f^{*}(G)=\sum_{i<j}d_{i}d_{j}r_{ij}$. This index is now known as multiplicative degree-Kirchhoff index.  Another weighted form of Kirchhoff index is the addictive degree-Kirchhoff index \cite{007}. Multiplicative degree-Kirchhoff index is closely related to the spectrum of the normalized Laplacian matrix $\mathcal{L}(G)$. For a connected graph $G$ of order $n$ and size $m$, Chen and Zhang \cite{005} showed that
$${K\!f}^{*}(G)=2m\sum^{n}_{i=2}\frac{1}{\lambda_i},$$
where $0=\lambda_1<\lambda_2\leqslant\cdots\leqslant\lambda_n$ $(n\geqslant2)$ are the eigenvalues of $\mathcal{L}(G)$.

Some techniques to determine the Kirchhoff index and multiplicative degree-Kirchhoff index were given in \cite{001,002,cle1,cle2,017}. Some other topics on the Kirchhoff index  and the multiplicative degree-Kirchhoff index of a graph  may be referred to \cite{018,019,023,025,026} and references therein.  
In the last decades, many researchers are devoted to give closed formulas for the  Kirchhoff index and the multiplicative degree-Kirchhoff index of graphs with special structures, such as cycles \cite{012},   ladder graphs \cite{cin}, ladder-like chains \cite{004,014}, liner phenylenes \cite{li0,peng,zhu}, linear polyomino chains \cite{010,022}, linear pentagonal chains \cite{he,ywang}, linear hexagonal chains \cite{009,li1,022},   linear octagonal-quadrilateral networks \cite{liu1}, linear crossed chains \cite{pan,pan1}, and composite graphs \cite{025},. 

Given two graphs $G$ and $H$, \textit{the strong product} of $G$ and $H$, denoted by $G\boxtimes H$, is the graph with vertex set $V(G)\times V(H)$, where two distinct vertices $(u_1,v_1)$ and $(u_2,v_2)$ are adjacent whenever $u_1$ and $u_2$ are equal or adjacent in $G$, and $v_1$ and $v_2$ are equal or adjacent in $H$.   Pan and Li \cite{pan1} determined the Kirchhoff index, multiplicative degree-Kirchhoff index and number of spanning trees of  graph $P_2\boxtimes P_n$.

 Motivated by \cite{pan1}, we consider the graph $P_2\boxtimes C_n$, where $n\geqslant3$. Let $G_n=P_2\boxtimes C_n$, where the graph $G_n$ is depicted in Figure 1. Obviously,  $|V{(G_n)}|=2n$ and $|E{(G_n)}|=5n$. Let $E^{\prime}$ be the set of vertical edges of $G_n$, where $E^{\prime}=\{ii^{\prime}:i=1,2,\dots,n\}$.     Let $\mathcal{G}^r_n$ be the set of subgraphs obtained from $G_n$ by deleting $r$ vertical edges of $G_n$, where $0\leqslant r\leqslant n$. It is easy to obtain that $\mathcal{G}^0_n=\{G_n\}$.

\begin{figure}[htbp]
	\centering 
	\includegraphics[height=2.      in, width=3   in,angle=0]{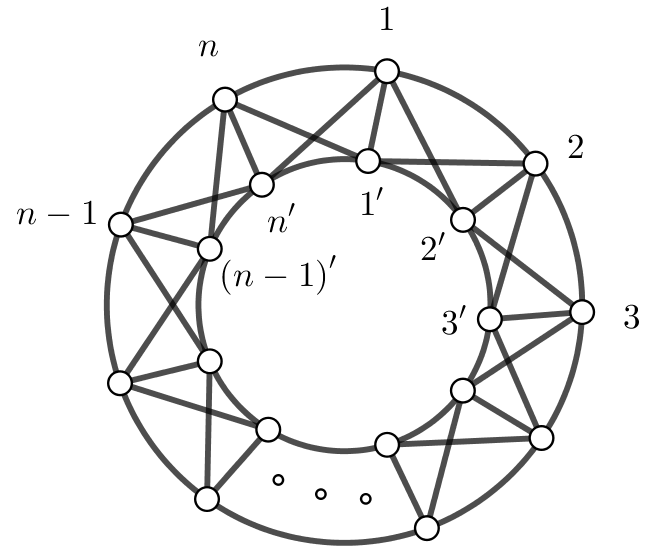}
	\caption{Graph $G_n$ with some labelled vertices. }
	\label{fig01}
\end{figure}

In this paper, explicit expressions for the Kirchhoff index, the multiplicative degree-Kirchhoff index and the number of spanning trees of $G_n$ are determined, respectively. It is nice to find that the Kirchhoff (resp. multiplicative degree-Kirchhoff) index of $G_n$ is almost one-sixth of its Wiener (resp. Gutman) index. Additionally, for any graph $G^r_n\in \mathcal{G}^r_{n}$,  its Kirchhoff index and  number of spanning trees are completely determined, respectively. Moreover, the Kirchhoff index of $G^r_n$ is shown to be almost one-sixth of its Wiener index.

\section{Preliminaries}
 For convenience, let $V_1=\{1,2,\ldots,n\}$ and $V_2=\{1^\prime,2^\prime,\ldots,n^\prime\}$. Then $L(G_n)$ and $\mathcal{L}(G_n)$ can be written as below:
\begin{align*}
	&\hspace{1cm}{L}(G_n)=\left(
	\begin{array}{cc}
		{L}_{{11}} & {L}_{{12}} \\
		{L}_{{21}} & {L}_{{22}} 
	\end{array}
	\right),
	&
	&\mathcal{L}(G_n)=\left(
	\begin{array}{cc}
		\mathcal{L}_{{11}} & \mathcal{L}_{{12}} \\
		\mathcal{L}_{{21}} & \mathcal{L}_{{22}} 
	\end{array}
	\right),
	&
\end{align*}
where ${L}_{{ij}}$ (resp. $\mathcal{L}_{{ij}}$) denotes the submatrix whose rows correspond  to vertices in $V_i$ and columns correspond to vertices in $V_j$. By the construction of $G_n$, one can easily verify that $L_{11}=L_{22}$, ${L}_{12}=L_{21}$, $\mathcal{L}_{11}=\mathcal{L}_{22}$ and $\mathcal{L}_{12}=\mathcal{L}_{21}$. Let
$$T=\left(
	\begin{array}{cc}
		\frac{1}{\sqrt2}I_{n} & \frac{1}{\sqrt2}I_{n} \\
		\frac{1}{\sqrt2}I_{n} & -\frac{1}{\sqrt2}I_{n} 
	\end{array}
	\right).$$
Then
\begin{align*}
	&\hspace{1cm}T{L}(G_n)T=\left(
	\begin{array}{cc}
		{L}_{A} & 0 \\
		0 & {L}_{S} 
	\end{array}
	\right),
	&
	&T\mathcal{L}(G_n)T=\left(
	\begin{array}{cc}
		\mathcal{L}_{A} & 0 \\
		0 & \mathcal{L}_{S} 
	\end{array}
	\right),
	&
\end{align*}
where ${L}_A={L}_{{11}}+{L}_{{12}}$, ${L}_S={L}_{{11}}-{L}_{{12}}$,  $\mathcal{L}_A=\mathcal{L}_{{11}}+\mathcal{L}_{{12}}$ and $\mathcal{L}_S=\mathcal{L}_{{11}}-\mathcal{L}_{{12}}$.

Let $\Phi(B):=\det(xI-B)$ be the characteristic polynomial of a square matrix $B$, where $I$ is an unit matrix with the same order as that of $B$. Similar to the decomposition theorem obtained in \cite{008,pan}, we can obtain the decomposition theorem of $G_n$ as below. Here, we omit the proof.

\begin{lemma}
	\label{lem21}
	Let ${L}_A$, ${L}_S$,  $\mathcal{L}_A$ and $\mathcal{L}_S$ be defined as above. Then
	\begin{flalign*}
	&\hspace{1cm}\Phi(L(G_n))={\Phi({L}_A)}\cdot{\Phi({L}_S)},&  &\Phi(\mathcal{L}(G_n))={\Phi(\mathcal{L}_A)}\cdot{\Phi(\mathcal{L}_S)}.&
	\end{flalign*}
\end{lemma}

\begin{lemma}\cite{013}\label{kf}
	Let $G$ be an $n$-vertex connected graph. Then $K\!f(G)=n\sum^n_{i=2}\frac{1}{\rho_i}$. 
\end{lemma}

\begin{lemma}\cite{005}\label{kff}
	Let $G$ be an $n$-vertex connected graph with $m$ edges. Then $K\!f^{*}(G)=2m\sum^n_{i=2}\frac{1}{\lambda_i}$. 
\end{lemma}
 
 \begin{lemma}\label{the22}
 	\cite{006}
 	Let $G$ be an $n$-vertex connected graph. Then $\tau(G)=\frac{1}{n}\prod^n_{i=2}\rho_i$, where $\tau(G)$ is the number of spanning trees of $G$.	
 \end{lemma}
 


\section{Resistance distance-based graph invariants and number of spanning trees of $G_n$} 

In this section, we will determine the Kirchhoff index, multiplicative degree-Kirchhoff index and number of spanning trees of $G_n$. The Laplacian matrix $L(C_n)$ of $C_n$ is very important in this section. It was proved in \cite{ande} that the eigenvalues of $L(C_n)$ are $4\sin^2(\frac{\pi i}{n})$, where $i=1,2,\ldots,n$. In the rest of this paper, let $\alpha_i:=4\sin^2(\frac{\pi i}{n})$. Then $\alpha_n=0$ and $\alpha_i>0$ for $i=1,2,\ldots,n-1$.


\subsection{Kirchhoff index and number of spanning trees of $G_n$ }
In this subsection, we will give the explicit formulas for the Kirchhoff index and number of spanning trees of $G_n$. Moreover, we will prove that the Kirchhoff index of $G_n$ is almost one-sixth of its Wiener index. First, one can see that
$${L}_{{11}}=\left(
	\begin{array}{ccccccc}
		5& -1& 0 & 0& \dots& 0& -1\\
		-1& 5& -1 & 0& \dots& 0& 0\\
		0& -1 & 5& -1 &  \dots&  0& 0\\
		0& 0& -1 & 5&  \dots& 0& 0 \\
		\vdots& \vdots& \vdots& \vdots& \ddots& \vdots& \vdots\\
		0& 0& 0& 0 & \dots&  5& -1\\
	-1& 0& 0& 0 &    \dots&  -1& 5
	\end{array}
	\right)_{n\times{n}}$$
and
$${L}_{{12}}=\left(
\begin{array}{ccccccc}
-1& -1& 0 & 0& \dots& 0& -1\\
-1& -1& -1 & 0& \dots& 0& 0\\
0& -1 & -1& -1 &  \dots&  0& 0\\
0& 0& -1 & -1&  \dots& 0& 0 \\
\vdots& \vdots& \vdots& \vdots& \ddots& \vdots& \vdots\\
0& 0& 0& 0 & \dots&  -1& -1\\
-1& 0& 0& 0 &    \dots&  -1& -1
\end{array}
\right)_{n\times{n}}.$$
Since $L_{A}=L_{11}+L_{12}$ and $L_{S}=L_{11}-L_{12}$, then
$${L}_{{A}}=\left(
\begin{array}{ccccccc}
4& -2& 0 & 0& \dots& 0& -2\\
-2& 4& -2 & 0& \dots& 0& 0\\
0& -2 & 4& -2 &  \dots&  0& 0\\
0& 0& -2 & 4&  \dots& 0& 0 \\
\vdots& \vdots& \vdots& \vdots& \ddots& \vdots& \vdots\\
0& 0& 0& 0 & \dots&  4& -2\\
-2& 0& 0& 0 &    \dots&  -2& 4
\end{array}
\right)_{n\times{n}}$$
and
$${L}_{{S}}=\left(
\begin{array}{ccccccc}
6& 0& 0 & 0& \dots& 0& 0\\
0& 6& 0 & 0& \dots& 0& 0\\
0& 0 & 6& 0 &  \dots&  0& 0\\
0& 0& 0 & 6&  \dots& 0& 0 \\
\vdots& \vdots& \vdots& \vdots& \ddots& \vdots& \vdots\\
0& 0& 0& 0 & \dots& 6& 0\\
0& 0& 0& 0 &    \dots&  0& 6
\end{array}
\right)_{n\times{n}}.$$
Note that ${L}_A=2L(C_n)$ and ${L}_S$ is a diagonal matrix. By Lemma \ref{lem21}, $2\alpha_1,2\alpha_2,\ldots,2\alpha_n,6,6,\ldots,6$ are all the eigenvalues of ${L}{(G_n)}$, i.e. $0,2\alpha_1,2\alpha_2,\ldots,2\alpha_{n-1},6,6,\ldots,6$ are all the eigenvalues of ${L}{(G_n)}$. Then we can get the following theorem.

\begin{theorem}\label{thmm1}
For $n\geq3$, let $G_n=P_2\boxtimes C_n$. Then

$(1)$ $K\!f(G_n)=\frac{n^3+4n^2-n}{12}.$

$(2)$ $\tau(G_n)=n\cdot2^{2n-2}\cdot3^n.$

$(3)$ $\lim_{n\to\infty}\frac{K\!f(G_n)}{W(G_n)}=\frac{1}{6}.$
\end{theorem}

\begin{proof}
$(1)$ Note that $|V(G_n)|=2n$ and $K\!f(C_n)=\frac{n^3-n}{12}$ (see \cite{024}). By Lemma \ref{kf}, we have
	\begin{align*}
		\hspace{1.5cm}K\!f(G_n)&=2n\bigg(\sum^{n-1}_{i=1}\frac{1}{2\alpha_i}+\frac{n}{6}\bigg)\nonumber&\\
		&=n\sum^{n-1}_{i=1}\frac{1}{\alpha_i}+\frac{n^2}{3}\nonumber&\\
		&=K\!f(C_n)+\frac{n^2}{3}\nonumber&\\
		&=\frac{n^3+4n^2-n}{12}.&
	\end{align*}
	
	$(2)$ It follows from Lemma \ref{the22} that
	\begin{align*}
		\hspace{1.5cm}\tau(G_{n})&=\frac{1}{2n}\prod^{n-1}_{i=1}(2\alpha_i)\cdot6^n\nonumber&\\
		&=2^{n-2}\cdot6^n\cdot\frac{1}{n}\prod^{n-1}_{i=1}\alpha_i\nonumber&\\
		&=2^{n-2}\cdot6^n\cdot\tau(C_n)\nonumber&\\
		&=n\cdot2^{2n-2}\cdot3^{n}.&
	\end{align*}	
	
	$(3)$ We now calculate the value of $W(G_n)$. If $n$ is odd, for each vertex $i$ in $G_n$, we have $$f(i,n)=1+4\sum^{\frac{n-1}{2}}_{k=1}{k}=\frac{n^2+1}{2}.$$
	\begin{table}[!ht] \centering \caption{{Kirchhoff index and number of spanning trees of graphs from $G_3$ to $G_{11}$.} }
		\begin{tabular}{lllllllll} \hline
			G& $K\!f(G)$& $\tau(G)$& 	G& $K\!f(G)$& $\tau(G)$& 	G& $K\!f(G)$& $\tau(G)$\\ \hline
			$G_3$ & 5.00& 1296& $G_{6}$& 29.50& 4478976& $G_{9}$ & 87.00& 11609505792\\
			$G_4$ & 10.33& 20736& $G_{7}$& 44.33& 62705664& $G_{10}$ & 115.83& 154793410560\\
			$G_5$ & 18.33& 311040& $G_{8}$& 63.33& 859963392& $G_{11}$ & 150.33& 2043273019392\\	
			\hline
		\end{tabular}
		\label{tab2111}
	\end{table}
Then	
	$$W(G_n)=\frac{2\sum^{n}_{i=1}f(i,n)}{2}=\sum^{n}_{i=1}f(i,n)=\frac{n^3+n}{2}.$$
If $n$ is even, for each vertex $i$ in $G_n$, we have $$f(i,n)=1+4\sum^{\frac{n-2}{2}}_{k=1}{k}+2\cdot\frac{n}{2}=\frac{n^2+2}{2}.$$
Then	
$$W(G_n)=\frac{2\sum^{n}_{i=1}f(i,n)}{2}=\sum^{n}_{i=1}f(i,n)=\frac{n^3+2n}{2}.$$	
Note that $K\!f(G_n)=\frac{n^3+4n^2-n}{12}$. Hence, we have	$\lim_{n\to\infty}\frac{K\!f(G_n)}{W(G_n)}=\frac{1}{6}$ as desired.
\end{proof}

Kirchhoff index and number  of spanning trees of graphs from $G_3$ to $G_{11}$ are listed in Table 1, respectively.

 \subsection{Multiplicative degree-Kirchhoff index of  $G_n$}

In this subsection, we will determine the multiplicative degree-Kirchhoff index of $G_n$ and show that the multiplicative degree-Kirchhoff index of $G_n$ is almost one-sixth of its Gutman index. Note that
$$\mathcal{L}_{{11}}=\left(
\begin{array}{ccccccc}
1& -\frac{1}{5}& 0 & 0& \dots& 0& -\frac{1}{5}\\
-\frac{1}{5}& 1& -\frac{1}{5}& 0& \dots& 0& 0\\
0& -\frac{1}{5}& 1& -\frac{1}{5} &  \dots&  0& 0\\
0& 0& -\frac{1}{5} & 1&  \dots& 0& 0 \\
\vdots& \vdots& \vdots& \vdots& \ddots& \vdots& \vdots\\
0& 0& 0& 0 & \dots&  1& -\frac{1}{5}\\
-\frac{1}{5}& 0& 0& 0 &    \dots&  -\frac{1}{5}& 1
\end{array}
\right)_{n\times{n}}$$
and
$$\mathcal{L}_{{12}}=\left(
\begin{array}{ccccccc}
-\frac{1}{5}& -\frac{1}{5}& 0 & 0& \dots& 0& -\frac{1}{5}\\
-\frac{1}{5}& -\frac{1}{5}& -\frac{1}{5}& 0& \dots& 0& 0\\
0& -\frac{1}{5}& -\frac{1}{5}& -\frac{1}{5} &  \dots&  0& 0\\
0& 0& -\frac{1}{5} & -\frac{1}{5}&  \dots& 0& 0 \\
\vdots& \vdots& \vdots& \vdots& \ddots& \vdots& \vdots\\
0& 0& 0& 0 & \dots&  -\frac{1}{5}& -\frac{1}{5}\\
-\frac{1}{5}& 0& 0& 0 &    \dots&  -\frac{1}{5}& -\frac{1}{5}
\end{array}
\right)_{n\times{n}}.$$
\begin{table}[!ht] \centering \caption{{Multiplicative degree-Kirchhoff index  of graphs from $G_3$ to $G_{15}$.} }
	\begin{tabular}{llllllll} \hline
		G& $K\!f^*(G)$& 	G& $K\!f^*(G)$& 	G& $K\!f^*(G)$& 	G& $K\!f^*(G)$\\ \hline
		$G_3$ & 125.00&  $G_{6}$& 737.50&  $G_{9}$ & 2175.00& $G_{13}$ & 5958.33 \\
		$G_4$ & 258.33&   $G_{7}$& 1108.33&  $G_{10}$ & 2895.33& $G_{14}$& 7320.83\\
		$G_5$ & 458.33&   $G_{8}$& 1583.33&  $G_{11}$ &   3758.33& $G_{15}$& 8875.00\\	
		\hline
	\end{tabular}
	\label{tab21}
\end{table}
Since $\mathcal{L}_A=\mathcal{L}_{{11}}+\mathcal{L}_{{12}}$ and $\mathcal{L}_S=\mathcal{L}_{{11}}-\mathcal{L}_{{12}}$, then we have
$$\mathcal{L}_{A}=\left(
\begin{array}{ccccccc}
\frac{4}{5}& -\frac{2}{5}& 0 & 0& \dots& 0& -\frac{2}{5}\\
-\frac{2}{5}& \frac{4}{5}& -\frac{2}{5}& 0& \dots& 0& 0\\
0& -\frac{2}{5}& \frac{4}{5}& -\frac{2}{5} &  \dots&  0& 0\\
0& 0& -\frac{2}{5} & \frac{4}{5}&  \dots& 0& 0 \\
\vdots& \vdots& \vdots& \vdots& \ddots& \vdots& \vdots\\
0& 0& 0& 0 & \dots&  \frac{4}{5}& -\frac{2}{5}\\
-\frac{2}{5}& 0& 0& 0 &    \dots&  -\frac{2}{5}& \frac{4}{5}
\end{array}
\right)_{n\times{n}}$$
and
$$\mathcal{L}_{S}=\left(
\begin{array}{ccccccc}
\frac{6}{5}& 0& 0 & 0& \dots& 0& 0\\
0& \frac{6}{5}& 0& 0& \dots& 0& 0\\
0& 0& \frac{6}{5}& 0 &  \dots&  0& 0\\
0& 0& 0& \frac{6}{5}&  \dots& 0& 0 \\
\vdots& \vdots& \vdots& \vdots& \ddots& \vdots& \vdots\\
0& 0& 0& 0 & \dots&  \frac{6}{5}& 0\\
0& 0& 0& 0 &    \dots&  0& \frac{6}{5}
\end{array}
\right)_{n\times{n}}.$$
Note that $\mathcal{L}_A=\frac{2}{5}L(C_n)$ and $\mathcal{L}_S$ is a diagonal matrix. By Lemma \ref{lem21}, $\frac{2}{5}\alpha_1,\frac{2}{5}\alpha_2,\ldots,\frac{2}{5}\alpha_n,\frac{6}{5},\frac{6}{5},\ldots,\frac{6}{5}$ are all the eigenvalues of $\mathcal{L}{(G_n)}$, i.e. $0,\frac{2}{5}\alpha_1,\frac{2}{5}\alpha_2,\ldots,\frac{2}{5}\alpha_{n-1},\frac{6}{5},\frac{6}{5},\ldots,\frac{6}{5}$ are all the eigenvalues of $\mathcal{L}{(G_n)}$. Then we can get the following theorem.
\begin{theorem}
	For $n\geq3$, let $G_n=P_2\boxtimes C_n$. Then
	
	$(1)$ $K\!f^{*}(G_n)=\frac{25n^3+100n^2-25n}{12}.$
	
	$(2)$ $\lim_{n\to\infty}\frac{K\!f^{*}(G_n)}{Gut(G_n)}=\frac{1}{6}.$
\end{theorem}
\begin{proof}
	$(1)$ 	Since $|E(G_n)|=5n$ and $K\!f(C_n)=\frac{n^3-n}{12}$, by Lemma \ref{kff}, we have
	\begin{align*}
	\hspace{1.5cm}K\!f^*(G_n)&=10n\bigg(\sum^{n-1}_{i=1}\frac{5}{2\alpha_i}+\frac{5n}{6}\bigg)\nonumber&\\
	&=25n\sum^{n-1}_{i=1}\frac{1}{\alpha_i}+\frac{25n^2}{3}\nonumber&\\
	&=25K\!f(C_n)+\frac{25n^2}{3}\nonumber&\\
	&=\frac{25n^3+100n^2-25n}{12}.&
	\end{align*}
	
	$(2)$ It is easy to calculate that $\textrm{Gut}(G_n)=25W(G_n)$ since  $G_n$ is regular. Note that $K\!f^{*}(G_n)=\frac{25n^3+100n^2-25n}{12}$. Together with Theorem \ref{thmm1}, we have $\lim_{n\to\infty}\frac{K\!f^{*}(G_n)}{\textrm{Gut}(G_n)}=\frac{1}{6}$ as desired.
\end{proof}


Multiplicative degree-Kirchhoff index of graphs from $G_3$ to $G_{15}$ are listed in Table 2.

\section{Kirchhoff index and number of spanning trees of $G^r_n$}
In this section, we will determine the Kirchhoff index and number of spanning trees for any graph $G^r_n\in \mathcal{G}^r_{n}$. Moreover, we will show that the Kirchhoff index of $G^r_n$ is nearly one-sixth of its Wiener index. 

Similar to Lemma \ref{lem21}, we can obtain that the  spectrum of $L(G^r_n)$ consists of the eigenvalues of both $L_A(G^r_n)$ and $L_S(G^r_n)$. Let $d_i$ be the degree of vertex $i$ in $G^r_n$. Then $d_i=4$ or $d_i=5$ in $G^r_n$. It is routine to check that
$${L}_{{11}}(G^r_n)=\left(
\begin{array}{ccccccc}
d_1& -1& 0 & 0& \dots& 0& -1\\
-1& d_2& -1 & 0& \dots& 0& 0\\
0& -1 & d_3& -1 &  \dots&  0& 0\\
0& 0& -1 & d_4&  \dots& 0& 0 \\
\vdots& \vdots& \vdots& \vdots& \ddots& \vdots& \vdots\\
0& 0& 0& 0 & \dots&  d_{n-1}& -1\\
-1& 0& 0& 0 &    \dots&  -1& d_n
\end{array}
\right)_{n\times{n}}$$
and
$${L}_{{12}}(G^r_n)=\left(
\begin{array}{ccccccc}
t_1& -1& 0 & 0& \dots& 0& -1\\
-1& t_2& -1 & 0& \dots& 0& 0\\
0& -1 & t_3& -1 &  \dots&  0& 0\\
0& 0& -1 & t_4&  \dots& 0& 0 \\
\vdots& \vdots& \vdots& \vdots& \ddots& \vdots& \vdots\\
0& 0& 0& 0 & \dots&  t_{n-1}& -1\\
-1& 0& 0& 0 &    \dots&  -1& t_n
\end{array}
\right)_{n\times{n}}.$$
where $t_{i}=0$ if $d_{i}=4$ and $t_{i}=-1$ if $d_{i}=5$. Then for any graph $G^r_n\in \mathcal{G}^r_{n}$, $d_{i}+t_i=4$ holds for all $1\leqslant i\leqslant n$. Since $L_A(G^r_n)=L_{11}(G^r_n)+L_{12}(G^r_n)$ and $L_S(G^r_n)=L_{11}(G^r_n)-L_{12}(G^r_n)$, then 
$${L}_{{A}}(G^r_n)=\left(
\begin{array}{ccccccc}
4& -2& 0 & 0& \dots& 0& -2\\
-2& 4& -2 & 0& \dots& 0& 0\\
0& -2 & 4& -2 &  \dots&  0& 0\\
0& 0& -2 & 4&  \dots& 0& 0 \\
\vdots& \vdots& \vdots& \vdots& \ddots& \vdots& \vdots\\
0& 0& 0& 0 & \dots&  4& -2\\
-2& 0& 0& 0 &    \dots&  -2& 4
\end{array}
\right)_{n\times{n}}$$
and
$${L}_{{S}}(G^r_n)=\left(
\begin{array}{ccccccc}
s_1& 0& 0 & 0& \dots& 0& 0\\
0& s_2& 0 & 0& \dots& 0& 0\\
0& 0 & s_3& 0 &  \dots&  0& 0\\
0& 0& 0 & s_4&  \dots& 0& 0 \\
\vdots& \vdots& \vdots& \vdots& \ddots& \vdots& \vdots\\
0& 0& 0& 0 & \dots& s_{n-1}& 0\\
0& 0& 0& 0 &    \dots&  0& s_n
\end{array}
\right)_{n\times{n}}.$$
where $s_{i}=4$ if $d_{i}=4$ and $s_{i}=6$ if $d_{i}=5$.

Note that ${L}_A(G^r_n)=2L(C_n)$ and ${L}_S$ is a diagonal matrix. Then,  $2\alpha_1,2\alpha_2,\ldots,2\alpha_n,s_1,s_2,\ldots,s_n$ are all the eigenvalues of ${L}{(G^r_n)}$, i.e. $0,2\alpha_1,2\alpha_2,\ldots,2\alpha_{n-1},s_1,s_2,\ldots,s_n$ are all the eigenvalues of ${L}{(G^r_n)}$. Next, we can get the following theorem.

\begin{theorem}\label{thm1}
	For any graph $G^r_{n}\in\mathcal{G}^r_{n}$ with $n\geqslant3$, we have
	
	$(1)$ $K\!f(G^r_n)=\frac{n^3+4n^2+(2r-1)n}{12}.$
	
	$(2)$ $\tau(G^r_n)=n\cdot2^{2n+r-2}\cdot3^{n-r}.$
	
	$(3)$ $\lim_{n\to\infty}\frac{K\!f(G^r_n)}{W(G^r_n)}=\frac{1}{6}.$
\end{theorem}

\begin{proof}
		For any graph $G^r_{n}\in\mathcal{G}^r_{n}$ with $n\geq3$, without loss of generality, assume that $s_1=s_2=\cdots=s_r=4$ and $s_{r+1}=s_{r+2}=\cdots=s_n=6$.

	$(1)$ Note that $|V(G^r_n)|=2n$ and $K\!f(C_n)=\frac{n^3-n}{12}$. By Lemma \ref{kf}, we have
	\begin{align*}
	\hspace{1.5cm}K\!f(G^r_n)&=2n\bigg(\sum^{n-1}_{i=1}\frac{1}{2\alpha_i}+\frac{n-r}{6}+\frac{r}{4}\bigg)\nonumber&\\
	&=n\sum^{n-1}_{i=1}\frac{1}{\alpha_i}+\frac{n(n-r)}{3}+\frac{nr}{2}\nonumber&\\
	&=K\!f(C_n)+\frac{2n^2+nr}{6}\nonumber&\\
	&=\frac{n^3+4n^2+(2r-1)n}{12}.&
	\end{align*}
	
	$(2)$ By Lemma \ref{the22}, we obtain
	\begin{align*}
	\hspace{1.5cm}\tau(G^r_{n})&=\frac{1}{2n}\prod^{n-1}_{i=1}(2\alpha_i)\cdot4^r\cdot6^{n-r}\nonumber&\\
	&=2^{2n+r-2}\cdot3^{n-r}\cdot\frac{1}{n}\prod^{n-1}_{i=1}\alpha_i\nonumber&\\
	&=2^{2n+r-2}\cdot3^{n-r}\cdot\tau(C_n)\nonumber&\\
	&=n\cdot2^{2n+r-2}\cdot3^{n-r}.&
	\end{align*}	
	
	$(3)$ It is easy to calculate that $W(G^r_n)=W(G_n)+r$. Since $K\!f(G^r_n)=\frac{n^3+4n^2+(2r-1)n}{12}$,  thus we have	$\lim_{n\to\infty}\frac{K\!f(G^r_n)}{W(G^r_n)}=\frac{1}{6}$.
\end{proof}
 
 \begin{remark}
 	If $r=0$, then $\mathcal{G}^0_n=\{G_n\}$. One can see that Theorem \ref{thmm1} is a corollary of Theorem \ref{thm1}.
 \end{remark}

\section{Concluding remarks}
In this paper, we first establish the explicit expressions for the Kirchhoff index, multiplicative degree-Kirchhoff index and number of spanning trees of $G_n$, where $G_n=P_2\boxtimes C_n$ and $n\geqslant3$. We find that the Kirchhoff (resp. multiplicative degree-Kirchhoff)  index is almost one-sixth of its Wiener (resp. Gutman) index. Later, we construct a family of graphs obtained from $G_n$ by deleting any $r$ vertical edges of $G_n$, and  show that the Kirchhoff indices of these graphs are almost one-sixth of their Gutman indices. It would be interesting to determine their multiplicative degree-Kirchhoff indices and Gutman indices. We will do it in the near future.

Motivated by the construction $P_2\boxtimes P_n$ in \cite{pan1} and $P_2\boxtimes C_n$ in this paper, we propose the following question for further study:
\begin{question}
For a simple connected graph $G$, how can we determine the Laplacian spectrum and normalized Laplacian specteum of the graph $P_2\boxtimes{G}$?
\end{question}

\section*{Acknowledgements} The authors would like to express their sincere gratitude to all the referees for their careful reading and insightful suggestions.





\end{document}